\documentclass{article}
\usepackage{latexsym,amsfonts,amsthm,amsmath,mathrsfs}
\usepackage{color}
\newtheorem{teo}{Theorem}[section]
\newtheorem{lemma}[teo]{Lemma}
\newtheorem{prop}[teo]{Proposition}
\newtheorem{cor}[teo]{Corollary} 
\theoremstyle{definition}
\newtheorem{defin}[teo]{Definition}
\newtheorem{rem}[teo]{Remark}

\newcommand{\R}{\mathbb{R}}
\newcommand{\C}{\mathbb{C}}
\newcommand{\Z}{\mathbb{Z}}
\newcommand{\N}{\mathbb{N}}
\newcommand{\la}{\lambda}
\newcommand{\vp}{\varphi}
\newcommand{\ve}{\varepsilon}

\newcommand{\bs}{\boldsymbol}

\newcommand{\beq}{\begin{equation}}
\newcommand{\eeq}{\end{equation}}

\newcommand{\D}{\mathscr{D}}
\newcommand{\MM}{\mathbf{M}}

\newcommand{\jj}{\mathbf{j}}

\newcommand{\EE}{\boldsymbol{E}}
\newcommand{\HHH}{\boldsymbol{H}}
\newcommand{\curl}{\nabla \times}
\newcommand{\A}{\mathbf{A}}

\title{Spinning Q--balls in Abelian Gauge Theories with \\ positive potentials: existence and non existence}
\author{{\sc Dimitri Mugnai}\\Dipartimento di Matematica e
Informatica\\Universit\`a di Perugia\\Via Vanvitelli 1, 06123
Perugia - Italy\\ tel. +39 075 5855043, fax. +39 075 5855024,\\
e-mail: dimitri.mugnai@unipg.it\\
\newline \\
{\sc Matteo Rinaldi}\\Department of Mathematical Sciences\\Carnegie Mellon University
\\5000 Forbes Avenue, Pittsburgh, PA 15213, USA\\ tel. +1 412 2688412,\\ e-mail: matteor@andrew.cmu.edu}

\date{}
\begin{document}

\maketitle
\begin{abstract}
We study the existence of cylindrically symmetric
electro-magneto-static solitary waves for a system of a nonlinear
Klein--Gordon equation coupled with Maxwell's equations in presence
of a positive mass and of a nonnegative nonlinear potential. Nonexistence results are provided as well.
\end{abstract}


2000AMS Subject Classification: 35J50, 81T13, 35Q40

\section{Introduction, motivations and results}

In recent years great attention has being paid to some classes
of systems of partial differential equations that provide a model
for the interaction of matter with electromagnetic field. Such
theories are known in literature as Abelian Gauge Theories, and in
this framework a crucial r\^ole is played by systems whose field
equation is the Klein--Gordon's one. In particular, we recall the
papers \cite{new2}, \cite{12}, \cite{bf}, \cite{BF}, \cite{11},
\cite{bfarc}, \cite{bh}, \cite{tdnonex}, \cite{teadim}, \cite{egs},
\cite{km}, \cite{long}, \cite{ms}, \cite{mug}, \cite{dm2},
\cite{new3} and \cite{vw}, where existence or non existence results
are proved in the whole physical space for systems of Klein-Gordon-Maxwell type.

Here we are interested in a particular class of solutions,
consisting in the so called \emph{solitary waves}, i.e. solutions of
a field equation whose energy travels as a localized packet. This
kind of solutions plays an important r\^ole in these theories
because of their relationship with \emph{solitons}. ``Soliton'' is
the name by which solitary waves are known when they exhibit some
strong form of stability; they appear in many situations of
mathematical physics, such as classical and quantum field theory,
nonlinear optics, fluid mechanics and plasma physics (for example
see \cite{dodd}, \cite{felsa} and \cite{raja}). Therefore, the first
step to prove the existence of solitons is to prove the existence of
solitary waves, as we will do.

Our starting point is the following system, obtained by the
interaction of a Klein--Gordon field with Maxwell's equations, which
is, therefore, a model for electrodynamics:
\begin{equation}\label{evola0}
\begin{cases}
&(\partial_t+iq\phi)^2\psi-(\nabla -iq{\bf A})^2\psi +W'(\psi)=0,\\
& {\rm div}(\partial_t {\bf A}+\nabla \phi)=q\left({\rm Im} \frac{\partial_t \psi}{\psi}+q\phi\right)|\psi|^2,\\
& \nabla \times (\nabla \times {\bf A})+\partial_t(\partial_t {\bf A}+\nabla \phi)=q\left({\rm Im} \frac{\partial_t \psi}{\psi}-q{\bf A}\right)|\psi|^2.
\end{cases}
\end{equation}
Here $\psi:\R^3\times \R \to \C$, $\phi:\R^3\to \R$ and ${\bf A}:\R^3\times \R\to \R^3$, see \cite{bf} for the derivation of the general system and for a detailed description of the physical
meaning of the unknowns.

We are interested in standing waves solutions of system \eqref{evola0}, under the assumption that $W$ possesses some good
invariants (necessary to be considered in Abelian Gauge Theories),
typically some conditions of the form
\[
W(e^{i\alpha}u)=W(u) \quad \mbox{ and }\quad
(W')(e^{i\alpha}u)=e^{i\alpha}W'(u)
\]
for any function $u$ and any $\alpha\in \R$. Thus we look for solutions having the
special form
\begin{align}
\psi(x,t)=u(x)e^{iS(x,t)}, u : \R^3 \to \R&, \ S(x,t)= S_0(x) - \omega t \in \R, \ \omega \in \R, \label{29a} \\
\partial_t \A &= 0, \ \partial_t \phi = 0. \label{30a}
\end{align}
In this way the previous system reads as
\begin{equation}\label{evola}\left\{
\begin{aligned}
&-\Delta u+|\nabla S-q{\bf
A}|^2u-\Big(\frac{\partial S}{\partial t}+q\phi\Big)^2u+W'(u)=0,\\
&\frac{\partial}{\partial t}\Big[\Big(\frac{\partial S}{\partial
t}+q\phi\Big)u^2\Big]-{\rm div}
[(\nabla S-q{\bf A})u^2]=0,\\
&{\rm div}\Big(\frac{\partial {\bf A}}{\partial t}+ \nabla
\phi\Big)= q \Big(\frac{\partial S}{\partial
t}+q\phi\Big)u^2,\\
&\nabla \times (\nabla \times {\bf A})+\frac{\partial}{\partial
t}\Big(\frac{\partial {\bf A}}{\partial t}+\nabla \phi\Big)=q(\nabla
S-q{\bf A})u^2,
\end{aligned}\right.
\end{equation}
where the equations are the matter equation, the charge continuity
equation, the Gauss equation and the Amp\`ere equation,
respectively.

Three different types of finite energy, stationary nontrivial
solutions can be considered:
\begin{itemize}
\item electrostatic solutions: $\A =0, \phi \neq 0$;
\item magnetostatic solutions: $\A \neq 0, \phi=0$;
\item electro-magneto-static solutions: $\A \neq 0, \phi \neq 0$.
\end{itemize}
Under suitable assumptions, all these types of solutions may exist.

Existence and nonexistence of electrostatic solutions for system
\eqref{evola} have been proved under different assumptions on $W$:
in \cite{tdnonex} and \cite{teadim} the following potential (or more general ones) has been considered:
\[
W(s) = \frac{1}{2}s^2 -
\frac{s^p}{p}, \ s \geq 0.
\]
In \cite{bf} the case $4<p<6$, in \cite{teadim} the
case $2<p<6$ and in \cite{tdnonex} the remaining cases are
studied.

In \cite{12} and \cite{dm2} the existence of electrostatic solutions
has been studied for the first time when the potential $W$ is {\em
nonnegative}. In particular the existence of radially symmetric,
electrostatic solutions has been analyzed in both papers, and it
turns out that all these solutions have zero angular momentum.

Here we are interested in electro-magneto-static solutions when $W
\geq 0$; in particular, we shall study the existence of vortices,
which are solutions with non vanishing angular momentum, namely
solutions with $S_0(x)=l\theta (x)$ - $\theta$ is the polar function in cylindrical coordinates -, i.e. of the form
\begin{equation} \label{38a}
\psi (t,x) = u(x)e^{i(l\theta (x)-\omega t)}, \ l \in
\mathbb{Z}\setminus \{ 0 \},
\end{equation}
and we will see that the angular momentum $\MM_m$ of the matter
field of a vortex does not vanish (see Remark \ref{n4}); this fact
justifies the name ``vortex''. These kinds of solutions are also known as \textit{spinning
Q--balls}; in this regard we recall the pioneering paper of Rosen
\cite{rosen} and of Coleman \cite{coleman}. Coleman was the first
to use the name \emph{Q--ball}, referring to
spherically symmetric solutions.Vortices in the nonlinear
Klein--Gordon--Maxwell equations with a {\em nonnegative} nonlinear term
$W(s)$ with $W(0) = 0$ are also considered in Physics literature
with the name of \textit{gauged spinning Q-balls}, the name balls being used even if they do not
exhibit a spherical symmetry, as in the case treated in this paper. More precisely, 
spinning axially symmetric Q-balls have been constructed by Volkov and Wohnert \cite{vw}, and have already been analysed also in \cite{afkp}, \cite{bh}, \cite{kkl} and \cite{kkls}. For a review of the problem of constructing classical field theory
solutions describing stationary vortex rings we refer to \cite{rv}, where applications in relativistic field theories and non-linear optics is presented.

However, in most of the previous considerations the existence of such solutions is discussed only qualitatively, so that almost no solutions
of this type are explicitly known. Indeed, the mathematical existence of spinning $Q$--balls was given for the first time in \cite{BF}, though some numerical results are known since \cite{LSWW}. Therefore, this paper is a contribution to an existence theory which is still at the very beginning.

By \eqref{38a}, system \eqref{evola} becomes
\begin{align}
&-\Delta u + \left[ |l\nabla \theta - q\A|^2 - (\omega - q \phi)^2 \right ]u + W'(u) = 0, \label{39a} \\
&-\Delta \phi = q(\omega - q\phi)u^2, \label{40a} \\
&\curl (\curl \A) = q(l\nabla \theta - q\A)u^2, \label{41a}
\end{align}
which is the Klein--Gordon--Maxwell system we have investigated. Moreover, though the system \eqref{39a}--\eqref{40a}--\eqref{41a} was
obtained by means of considerations on gauge invariance of $W$, from
a mathematical point of view we can also replace \eqref{39a} with
\[
-\Delta u + \left[ |l\nabla \theta - q\A|^2 - (\omega - q \phi)^2
\right ]u + W_u(x,u) = 0,
\]
i.e. we could let $W$ depend on the $x$--variable. More precisely,
in order to use our functional approach, we let $W$ depend on
$(\sqrt{x_1^2+x_2^2},x_3)$, but we do not require any positivity far
from 0, in contrast to the usual Ambrosetti--Rabinowitz condition.
We think that this fact is quite interesting, both from a
mathematical and a physical point of view: for example, it may
happen that the potential is inactive in some cylinder, or, even
more interestingly, out of a cylinder, as it happens where strong
magnetic potential are present in linear accelerators.

According to what just said, in the second section we will show a
new existence result for system
\eqref{39a}--\eqref{40a}--\eqref{41a} under general assumptions on
the nonnegative potential $W$. We were inspired by the approach of
\cite{BF}, and for this reason, the functional structure is the same
one of that article. However, our hypotheses on $W$ imply, in
particular, that the potential $W(s)$ might be 0 for values of $s$
different from 0, in contrast to all previous results, where the
potential $W$ was assumed to lye above a parabola. This corresponds
to the situation in which, for values of the unknown different from
0, there is no interaction among particles (see \cite{teadim},
\cite{dm2}).

Moreover, even more interestingly, we show the existence of solutions
for {\em all} possible values of the charge $q$. We believe this is
a very nice result, since for the first time in literature from the
seminal paper by Coleman \cite{coleman}, in which the charge was
supposed small, as in all the subsequent papers in our bibliography,
we give existence results {\em for all} values of the charge.

In conclusion, though our assumptions are weaker, our results are
stronger than those found so far.

\begin{rem}
If we consider the electrostatic case, i.e. $-\Delta u+W'(u)=0$,
calling ``rest mass'' of the particle $u$ the quantity
$$
\int_{\R^3} W(u)\,dx,
$$
see \cite{bfarc}, our assumptions on $W$ imply that we are dealing {\it a priori}
with systems for particles having {\em positive mass}, which is, of
course, the physical interesting case.
\end{rem}

Entering into details, we shall study system
\eqref{39a}--\eqref{40a}--\eqref{41a} under the following hypothesis
on the potential $W$:
\begin{itemize}
\item [$W1$)] $W(s) \geq 0$ for all $s \geq 0$;
\item [$W2)$] $W$ is of class $C^2$ with $W(0)=W'(0)=0, W''(0)=m^2 >
0$;
\item [$W3)$] setting
\begin{equation} \label{n1a} W(s)=\frac{m^2}{2}s^2 + N(s), \end{equation} we assume that
there exist positive constants $c_1, c_2, p,\ell,$ with $2<\ell \leq p <
6$, such that for all $s \geq 0$ there holds
\[
|N'(s)| \leq c_1 s^{\ell-1} + c_2s^{p-1}.
\]
\end{itemize}
Moreover, though we are interested in positive solutions, it is
convenient to extend $W$ to all of $\R$ setting
\[
W(s) = W(-s)  \mbox{ for every } s < 0.
\]

The system \eqref{39a}--\eqref{40a}--\eqref{41a} was introduced in
\cite{BF} assuming $W1),\,W2),\, W3)$ and the fundamental
requirement
\begin{equation}\label{loro}
\inf_{s>0}\left( \frac{W(s)}{\frac{m^2}{2}s^2} \right) < 1.
\end{equation}
We immediately see that assumption $W3)$ plus \eqref{loro} is
equivalent to require that there exists $s_0 > 0$ such that $N(s_0)
< 0$, the first step in the classical ``Berestycki--Lions''
approach. In this paper we will use an hypothesis different from
\eqref{loro}, which will let us prove our main result without any
restriction on the charge $q$, in contrast to all
previous results.

Indeed, we will
assume
\begin{itemize}
\item [$W4$)] there exist $\tau>2$,
\[
D\geq \begin{cases}
3(1+l^2)^{\frac{\tau -2}{2}}2^{3\tau/2-5}m^{4-\tau} & \mbox{ if }q\leq1,\\
3(1+l^2)^{\frac{\tau -2}{2}}2^{3\tau/2-5}m^{4-\tau}q^{3(\tau-2)} & \mbox{ if }q>1
\end{cases}
\]
and $\ve_0>0$ with $\ve_0=\ve_0(q)$ if $q>1$,
such that
\[
N(s)\leq -D|s|^\tau \quad \mbox{ for all }s\in[0,\ve_0].
\]
\end{itemize}

It is clear that functions of the type $N(s)=|s|^p-|s|^q$, $2<q<p$,
satisfy $W4)$. Of course, $W4)$ implies that  there exists $s_0 > 0$ such that $N(s_0)
< 0$, but $W4)$ permits to prove existence results for any $q>0$ and suitable potentials $W$, see Theorem \ref{main}.

\begin{rem}
We emphasize the fact that $W1)$ and $W4)$ together imply that $D$ cannot be as large as desired, since the condition $W\geq0$ forces $D$ to depend on $\ve_0$ and $q$. 
However, we remark that the parameter $\ve_0$ is allowed to depend on $q$ only when $q > 1$, hence $D$ does not depend on the charge $q$ when $q \leq1$, but it depends only on $m$ and $l$.
As a consequence, the class of admissible potentials does {\em not} depend on the value of the charge $q$, whenever $q \leq 1$.
\end{rem}

As usual, for physical reasons, we look for solutions having finite
energy, i.e. $(u,\phi, \A)\in H^1\times \mathcal{D}^1 \times \left(
\mathcal{D}^1  \right)^3$, where $H^1=H^1(\R^3)$ is the usual
Sobolev space, and $\mathcal{D}^1=\mathcal{D}^1(\R^3)$ is the
completion of ${\mathscr D}=C^\infty_C(\R^{3})$ with respect to the
norm $\|u\|_{\mathcal{D}^1}^2:=\int_{\R^3}|\nabla u|^{2}\,dx$ (see
Section \ref{functional} for the precise functional setting). \\

Before giving our main result, we remark that, as in
\cite{BF}, the parameter $\omega$ is an unknown of the problem.
\begin{teo} \label{main}
Assume $W1)$, $W2)$, $W3)$, $W4)$, let $l \in \mathbb{Z}$ and $q\geq
0$. Then system \eqref{39a}--\eqref{40a}--\eqref{41a} admits a
finite energy solution in the sense of distributions
$(u,\omega,\phi,\A), u\neq 0, \omega > 0$ such that
\begin{itemize}
\item the maps $u, \phi$ depend only on the variables
$r=\sqrt{x_1^2+x_2^2}$ and $x_3$;
\item
\[
\int_{\R^3}\frac{u^2}{r^2}\,dx\in \R;
\]
\item the magnetic potential $\A$ has the following form:
\begin{equation} \label{n3a}
\A = a(r,x_3)\nabla \theta = a(r,x_3) \left( \frac{x_2}{r^2}\mathbf{e_1} - \frac{x_1}{r^2}\mathbf{e_2} \right).
\end{equation}
\end{itemize}
If $q=0$, then $\phi =0,\, \A=\bs{0}$. If $q>0$, then $\phi \neq 0$.
Moreover, $\A \neq \bs{0}$ if and only if $l \neq 0$.
\end{teo}

\begin{rem} 
By definition, the angular momentum is the quantity which
is preserved by virtue of the invariance under space rotations of
the Lagrangian with respect to the origin. Using the gauge invariant
variables, we get:
\[
\MM = \MM_m + \MM_f,
\]
where
\[
\MM_m = \int_{\R^3}{\biggl[ -x \times (\nabla u \partial_t
u) + x \times \frac{\rho \jj}{q^2u^2} \biggr] \ dx}
\]
and
\[
\MM_f = \int_{\R^3}{x \times (\EE \times \HHH) \ dx}.
\]
Here $\MM_m$ refers to the ``matter field'' and $\MM_f$ to the
``electromagnetic field'', while $\rho$ and $\jj$ denote the
electric charge and the current density, respectively.

We will see below that the solution found in Theorem \ref{main} has
nontrivial angular momentum, see Remark \ref{n4}.
\end{rem}

\begin{rem}
When $l=0$ and $q>0$ the last part of Theorem \ref{main} states the
existence of electrostatic solutions, namely finite energy solutions
with $u \neq 0, \phi \neq 0$ and $\A =\bs{0}$. This result is a variant
of a recent ones (see \cite{12} and \cite{dm2}).
\end{rem}
Moreover, let us observe that under general assumptions on $W$,
magnetostatic solutions (i.e. with $\omega = \phi =0$) do not exist.
In fact also the following proposition is proved in \cite{BF}:
\begin{rem}[\cite{BF}, Proposition 8]
Assume that $W$ satisfies the assumptions $W(0) = 0$ and $W'(s)s$
$\geq 0$. Then \eqref{39a}, \eqref{40a}, \eqref{41a} has no
solutions with $\omega = \phi = 0$ (see \cite[Proposition 1.2]{spos} for a related result).
\end{rem}

In our setting, we are able to prove the following nonexistence results:
\begin{teo}\label{senzanome}
If $u$ is a finite energy solution of \eqref{39a} with
\[
\int_{\R^3} N(u)\,dx\in \R,
\]
and
\begin{itemize}
\item $\omega^2<m^2$ and either $N\geq 0$ or $N'(s)s\leq 6 N(s)$ for all $s\in \R$,\\ or
\item $N'(s)s\geq 2N(s)$ for all $s\in \R$,
\end{itemize}
then $u\equiv 0$.
\end{teo}
A natural consequence is the following
\begin{cor}
If $u\in L^p(\R^3)$ is a finite energy solution of
\eqref{39a}--\eqref{40a}--\eqref{41a}, and
\begin{itemize}
\item $\omega^2<m^2$ and
\[
N(u) =  \left\{
\begin{aligned}
\dfrac{|u|^p}{p}, \ \ & p \leq 6, \\
-\dfrac{|u|^p}{p}, \ \ & p \geq 6,
\end{aligned}\right.
\]
or
\item
\[
N(u) =  \left\{\begin{aligned}
\dfrac{|u|^p}{p}, \ \ & p \geq 2, \\
-\dfrac{|u|^p}{p}, \ \ & p \leq 2,
\end{aligned}\right.
\]
\end{itemize}
then $u\equiv 0$.
\end{cor}

\begin{rem}
Theorem \ref{senzanome} implies that, in general, in order to have
vortices with $N\geq0$ it is necessary to have a ``large''
frequency. We are not aware of similar results in the theory of
vortices, and we believe such a result can shed a new light on this
subject.
\end{rem}

In Section \ref{secfissa} we shall prove another existence result
concerning a different kind of solutions, namely solutions having
fixed $L^2$ norm. In general these solutions cannot be obtained from
the solutions found in Theorem \ref{main}, for example via a
rescaling argument, and we shall focus on the case $\int_{\R^3}u^2
dx = 1$, which corresponds to look for solutions having a density of
probability equal to 1. An analogous result could be obtained for
$\int_{\R^3}u^2 dx = c \in \R ^{+}$, but the physical meaning of
this kind of solutions is not clear to us. Indeed, in different
situations it may happen that if $\int_{\R^3}u^2=c$ is fixed a
priori, then solutions appear only for certain values of $c$: a
typical example is in the context of boson stars, when solutions
with fixed energy do exist if and only if $c<M_C$, the Chandrasekhar
limit mass (see \cite{LY} and \cite{pseudo}).

Our result is the following
\begin{prop} \label{vincolo}
Under the hypotheses of Theorem $\ref{main}$, there exists
$\mu\in\R$ and a solution in the sense of distributions for the
system
\[
\begin{aligned}
&-\Delta u + \left[ |l\nabla \theta - q\A|^2 - (\omega - q \phi)^2 \right ]u + W'(u) = \mu u,  \\
&-\Delta \phi = q(\omega - q\phi)u^2,  \\
&\curl (\curl \A) = q(l\nabla \theta - q\A)u^2,
\end{aligned}
\]
such that $\int_{\R^3}u^2 dx = 1.$ Moreover, if $\omega^2 \leq m^2$
and $ N'(s)s \geq 0$ for all $s\in \R$, then $\mu > 0$.
\end{prop}

Due to the presence of the multiplier $\mu$, we give the following
\begin{defin}\label{massa}
We call \emph{effective mass} of the system the quantity $\tilde m =
m^2 - \mu$.
\end{defin}

\section{Preliminary setting}

\subsection{Standing wave solutions and vortices}

Substituting \eqref{29a} and \eqref{30a} in \eqref{evola}, we get
the following equations in $\R^3$:
\begin{align}
&- \Delta u + \biggl[ |\nabla S_0 - q\A|^2 - (\omega - q \phi)^2  \biggl]u + W'(u) = 0, \label{31} \\
&- \text{div} \biggl[ (\nabla S_0 - q \A)u^2 \biggr] = 0, \label{32} \\
&- \Delta \phi = q(\omega - q\phi)u^2, \label{33} \\
&\curl (\curl \A) = q(\nabla S_0 - q\A)u^2. \label{34}
\end{align}
We  can easily observe that \eqref{32} follows from \eqref{34}: as a
matter of fact, applying the divergence operator to both sides of
\eqref{34}, we immediately get \eqref{32}. Then we are reduced to
study the system \eqref{31}--\eqref{33}--\eqref{34}.

We are interested in finite-energy solutions - the most relevant
physical case - i.e. solutions of system
\eqref{31}--\eqref{33}--\eqref{34} for which the following energy is
finite:
\begin{equation} \label{35}
\begin{aligned}
\mathcal{E}(u) =& \frac{1}{2} \int_{\R^3}{\biggl( |\nabla u|^2 + |\nabla \phi|^2 +
|\curl \A|^2 + (|\nabla S_0 - q\A|^2 + (\omega - q\phi)^2)u^2 \biggr)dx} \\
&+ \int_{\R^3}{W(u)dx}
\end{aligned}
\end{equation}

Furthermore, in order to study the behavior of some particular functional which will be introduced later on, it is useful to give the electric charge $Q$ a specific representation in terms of the solution $u$, as (see e.g. \cite{BF}, p.644)
 \begin{equation}\label{36}
Q = q\sigma , \end{equation} where
 \begin{equation} \label{37}
\sigma = \int_{\R^3}
{(\omega - q\phi)u^2 \ dx}.
 \end{equation}
However, our strategy will consist in fixing a real number $\sigma$ and then find a solution $u$ which turns out to verify \eqref{37}.

\begin{rem}
When $u=0$, the only finite energy gauge potentials which
solve \eqref{33}, \eqref{34} are the trivial ones $\A =\bs{0}, \phi =0$.
\end{rem}

In particular, following \cite{BF}, we shall look for solutions of
the above system which are known in literature as vortices. In order
to do that, we need some preliminaries. First, set
\[
\Sigma = \Big\{ (x_1,x_2,x_3) \in \R^3: x_1=x_2=0 \Big\},
\]
and define the map
\[
\begin{aligned}
&\theta : \R^3 \setminus \Sigma \rightarrow \frac{\R}{2\pi \mathbb{Z}}, \\
&\theta(x_1,x_2,x_3) = \text{Im}\log(x_1+ix_2).
\end{aligned}
\]
The following definition is crucial:
\begin{defin}
A finite energy solution $(u,S_0,\phi,\A)$ of
\eqref{31}--\eqref{33}--\eqref{34} is called
\textbf{\textit{vortex}} if $S_0 =
l\theta$ for some $l \in \Z\setminus \{0\}$.
\end{defin}
Of course, in this case, $\psi$ has the form
\begin{equation} \label{38}
\psi (t,x) = u(x)e^{i(l\theta (x)-\omega t)}, \ l \in
\mathbb{Z}\setminus \{ 0 \}.
\end{equation}

\begin{rem} \label{n4}
In \cite[Proposition 7]{BF} it was proved that if $(u,\omega,\phi,\A)$ is a non trivial, finite energy solution
of  \eqref{31}--\eqref{33}--\eqref{34}, then the angular
momentum $\MM_m$ has the expression
\begin{equation}
\label{n5} \MM_m = - \left[ \int_{\R^3}{(l-qa)(\omega - q
\phi)u^2dx} \right] \mathbf{e_3},
\end{equation}
and, if $l \neq 0$, it does not vanish. Hence, in this case,  the
name ``vortex'' is justified and by Theorem \ref{main} the existence
of a spinning $Q$--ball is guaranteed.
\end{rem}

Now, observe that $\theta \in C^{\infty} \left ( \R^3 \setminus
\Sigma, \frac{\R}{2\pi \mathbb{Z}} \right  )$, and, with an abuse of
notation, we set
\[
\nabla \theta (x) = \frac{x_2}{x_1^2 + x_2^2}\mathbf{e_1} - \frac{x_1}{x_1^2 + x_2^2}\mathbf{e_2},
\]
where $\mathbf{e_1}, \mathbf{e_2}, \mathbf{e_3}$ is the standard
frame in $\R^3$.

Using the {\it Ansatz} \eqref{38}, equations \eqref{31}, \eqref{33},
\eqref{34} give rise to equations \eqref{39a}, \eqref{40a},
\eqref{41a}, which is the Klein--Gordon--Maxwell system we shall
study from now on.

\begin{rem}
If $\A = \left( \dfrac{x_2}{x_1^2 + x_2^2}, - \dfrac{x_1}{x_1^2 +
x_2^2}, 0  \right)$, we obviously get $\curl \A = 0$. {\it
Viceversa}, if $\A$ is irrotational and it solves \eqref{41a}, then
$\A=\frac{l}{q}\nabla \theta$. In such a case, system
\eqref{39a}--\eqref{40a}--\eqref{41a} reduces to the one considered
in \cite{dm2}, where, by Theorem \ref{senzanome}, we can now say
that the nontrivial solution found therein is such that
$\omega^2\geq m^2$.
\end{rem}

\subsection{Functional approach}\label{functional}
We shall follow the functional approach of \cite{BF}, with minor
changes in some parts. Anyway, our main Theorem \ref{main} has been
proved thanks to  completely new results (see Lemma \ref{inf0} and
Proposition \ref{lemma17}), which let us avoid any bound on $q$, differently from \cite{BF}.

First, we denote by $L^p\equiv L^p(\R^3)$ ($1\leq
p<+\infty$) the usual Lebesgue space endowed with the norm
\[
\|u\|_p^p:=\int_{\R^3}|u|^p\,dx.
\]
We also recall the continuous embeddings
\begin{equation}\label{imme}
H^1(\R^3)\hookrightarrow D^1(\R^3)\hookrightarrow L^6(\R^3)\quad
\hbox{and} \quad H^1(\R^3)\hookrightarrow L^p(\R^3)\quad \;\;\forall\, p\in[2,6] ,
\end{equation}
being $6$ the critical exponent for the Sobolev embedding
$\mathcal{D}^1(\R^3)\hookrightarrow L^p(\R^3)$. Here $H^1 \equiv H^1(\R^3)$ denotes the usual Sobolev space with norm
\[
\|u\|^2_{H^1} = \int_{\R^3}{(|\nabla u|^2+u^2)dx}
\]
and $\mathcal{D}^1=\mathcal{D}^1(\R^3)$ is the completion of
${\mathscr D}=C^\infty_C(\R^{3})$ with respect to the norm
\[
\|u\|_{\mathcal{D}^1}^2:=\int_{\R^3}|\nabla u|^{2}\,dx,
\]
induced by the scalar product $(
u,v)_{\mathcal{D}^1}:=\int_{\R^3}\nabla u\cdot \nabla v\,dx$.

Moreover, we need the weighted Sobolev space $\hat{H}^1 \equiv
\hat{H}^1_l(\R^3)$, depending on a fixed integer $l$, whose norm is
given by
\[
\|u\|^2_{\hat{H}^1} = \int_{\R^3}{\left[ |\nabla u|^2 + \left(
1+\frac{l^2}{r^2} \right)u^2 \right]dx}, \ l \in \mathbb{Z},
\]
where $r = \sqrt{x_1^2+x_2^2}$. Clearly $\hat{H}^1=H^1$ if and only
if $l=0$. Moreover, it is not hard to see that
\begin{equation}\label{denso}
\mbox{$C^\infty_C(\R^3)\cap \hat H^1(\R^3)$ is dense in $\hat H^1(\R^3)$}.
\end{equation}

We set
\[
\begin{aligned}
&H = \hat{H}^1 \times \mathcal{D}^1 \times \left( \mathcal{D}^1  \right)^3, \\
& \|(u,\phi,\A)\|^2_H = \int_{\R^3}{ \left[ |\nabla u|^2+ \left(  1+\frac{l^2}{r^2}
 \right)u^2 + |\nabla \phi|^2 + |\nabla \A|^2 \right] dx}.
\end{aligned}
\]
We shall denote by $u=u(r,x_3)$ any real function in $\R^3$ which
depends only on the cylindrical coordinates $(r,x_3)$, and we set
\[
\mathscr{D}_\sharp = \Big\{ u \in \mathscr{D}: u=u(r,x_3)
\Big\}.
\]
Finally, we shall denote by $\mathcal{D}^1_\sharp$ the closure of
$\mathscr{D}_\sharp$ in the $\mathcal{D}^1$ norm and by
$\hat{H}^1_\sharp$ the closed subspace of $\hat{H}^1$ whose
functions are of the form $u=u(r,x_3)$.

Now, we consider the functional
\begin{equation} \label{56}
\begin{aligned}
J(u,\phi,\A) &= \frac{1}{2}\int_{\R^3}{\big[|\nabla u|^2 - |\nabla \phi|^2 + |\curl \A|^2\big] dx} \\
&+ \frac{1}{2} \int_{\R^3}{\left[ |l \nabla \theta - q\A|^2 - (\omega -q\phi)^2  \right]u^2 dx} + \int_{\R^3}{W(u) dx},
\end{aligned}
\end{equation} where $(u,\phi,\A) \in H$. Formally, equations \eqref{39a}, \eqref{40a} and
\eqref{41a} are the Euler--Lagrange equations of the functional $J$,
and, indeed, standard computations show that the following lemma
holds:
\begin{lemma}  \label{lemma9}
Assume that $W$ satisfies $W3)$. Then the functional $J$ is of class
$C^1$ on $H$ and equations \eqref{39a}, \eqref{40a} and \eqref{41a}
are its Euler--Lagrange equations.
\end{lemma}

By the above lemma it follows that any critical point $(u,\phi,\A)
\in H$ of $J$ is a weak solutions of system
\eqref{39a}--\eqref{40a}--\eqref{41a}, namely
\begin{align}
&\int_{\R^3}{\big[\nabla u \cdot \nabla v + \left[ |l\nabla \theta -
q\A|^2- (\omega - q \phi)^2 \right]uv + W'(u)v\big]
dx} =0 \ \forall\, v \in \hat{H}^1, \label{57} \\
& \int_{\R^3}{\big[\nabla \phi \cdot \nabla w - qu^2(\omega -q \phi)w\big]\,
dx}=0 \ \forall\, w \in \mathcal{D}^1, \label{58} \\
&\int_{\R^3}{\big[(\curl \A)\cdot(\curl \mathbf{V})-qu^2(l\nabla
\theta-q\A)\cdot \mathbf{V} \big]dx}=0 \ \forall\, \mathbf{V} \in
(\mathcal{D}^1)^3. \label{59}
\end{align}

\subsection{Solutions in the sense of distributions}
Since $\mathscr{D}$ is {\em not} contained in $\hat{H}^1$, a
solution $(u,\phi,\A) \in H$ of \eqref{57}, \eqref{58}, \eqref{59}
need not be a solution of \eqref{39a}, \eqref{40a}, \eqref{41a} in
the sense of distributions on $\R^3$. However, we will show that the
singularity of $\nabla \theta (x)$ on $\Sigma$ is removable in the
following sense:
\begin{teo} \label{th10}
Let $(u_0,\phi_0,\A_0) \in H, u_0 \geq 0$ be a solution of
\eqref{57}, \eqref{58}, \eqref{59} (i.e. a critical point of $J$).
Then $(u_0,\phi_0,\A_0)$ is a solution of system
\eqref{39a}--\eqref{40a}--\eqref{41a} in the sense of distributions,
namely
\begin{align}
&\int_{\R^3}{\big[\nabla u_0 \cdot \nabla v + \left[ |l\nabla \theta
- q\A_0|^2 - (\omega - q \phi_0)^2 \right]u_0v + W'(u_0)v\big] dx}
=0
\ \forall\, v \in \mathscr{D}, \label{60} \\
& \int_{\R^3}{\big[\nabla \phi_0 \cdot \nabla w - qu_0^2(\omega -q \phi_0)
w\big] dx}=0 \ \forall\, w \in \mathscr{D}, \label{61} \\
&\int_{\R^3}{\big[(\curl \A_0)\cdot(\curl \mathbf{V})-qu_0^2(l\nabla
\theta-q\A_0)\cdot \mathbf{V}\big] dx}=0\  \forall\, \mathbf{V} \in
(\mathscr{D})^3. \label{62}
\end{align}
\end{teo}
A proof of Theorem \ref{th10} was given in \cite{BF}.

Let us now remark that the presence of the term
$-\int_{\R^3}{|\nabla \phi|^2 dx}$ gives the functional $J$ a strong
indefiniteness, namely any nontrivial critical point of $J$ has
infinite Morse index. It turns out that a direct approach to finding
critical points for $J$ is very hard. For this reason, as usual in
this setting, it is convenient to introduce a \textit{reduced
functional}.

\subsection{The reduced functional}

Writing equation \eqref{40a} as
\begin{equation}\label{79}
-\Delta \phi + q^2u^2\phi=q\omega u^2,
\end{equation}
then we can verify that the following holds:
\begin{prop}[\cite{teadim}, Proposition 2.2]\label{ex}
For every $u \in H^1(\R^3)$, there exists a unique $\phi = \phi_u
\in \mathcal{D}^1$ which solves \eqref{79} and there exists $S>0$
such that
\begin{equation}\label{stima}
\|\phi_u\|\leq qS\|u\|_{12/5}^2 \mbox{ for every }u\in H^1(\R^3).
\end{equation}
\end{prop}

\begin{lemma}\label{mettere}
If $u \in \hat{H}^1_\sharp(\R^3)$, then the solution $\phi = \phi_u$
of \eqref{79} belongs to $\mathcal{D}^1_\sharp(\R^3)$.
\end{lemma}
The proof is an adaptation of the analogue in \cite{teadim} and is thus omitted.

By the lemma above, we can define the map
\begin{equation}\label{80}
u \in \hat{H}^1_\sharp(\R^3) \mapsto Z_{\omega}(u) = \phi_u \in
\mathcal{D}^1_\sharp \ \text{solution of \eqref{79}}.
\end{equation}
Since $\phi_u$ solves \eqref{79}, clearly we have
\begin{equation} \label{81}
d_{\phi}J(u,Z_{\omega}(u),\A)=0,
\end{equation}
where $J$ is defined in \eqref{56} and $d_{\phi}J$ denotes the
partial differential of $J$ with respect to $\phi$.

Following the lines of the proof of \cite[Proposition 2.1]{tdnonex},
using Lemma \ref{mettere}, we can easily prove the following result:
\begin{prop}\label{prop12}
The map $Z_{\omega}$ defined in \eqref{80} is of class $C^1$ and
\begin{equation}\label{sopra}
(Z_{\omega}'[u])[v]= 2q\left(\Delta - q^2u^2  \right) ^{-1} \left[
(q\phi_u -\omega)uv \right] \quad \forall\,u,v\in
\mathcal{D}^1_\sharp.
\end{equation}
\end{prop}

For $u \in H^1(\R^3)$, let $\Phi = \Phi_u$ be the solution of
\eqref{79} with $\omega = 1$; then $\Phi$ solves the equation
\begin{equation} \label{82}
 -\Delta \Phi_u + q^2u^2\Phi_u=q u^2,
\end{equation}
and clearly
\begin{equation} \label{83}
\phi_u = \omega \Phi_u.
\end{equation}
Now let $q > 0$; then, by maximum principle arguments, one can show
that for any $u \in H^1(\R^3)$ the solution $\Phi_u$ of \eqref{82}
satisfies the following estimate, first proved in \cite{mug}:
\begin{equation}\label{84}
0 \leq \Phi_u \leq \frac{1}{q}.
\end{equation}

Now, if $(u,\A) \in \hat{H}^1 \times \left(\mathcal{D}^1 \right)^3$,
we introduce the \textit{reduced action functional}
\[
\tilde{J}(u,\A) = J(u,Z_{\omega}(u),\A).
\]
Recalling that $J$ and the map $u \rightarrow Z_{\omega}(u) =
\phi_u$ are of class $C^1$ by Lemma \ref{lemma9} and Proposition
\ref{prop12}, respectively, also the functional $\tilde{J}$ is of
class $C^1$. Now, by using the chain rule and \eqref{81}, it is
standard to show that the following Lemma holds:
\begin{lemma} \label{85}
If $(u,\A)$ is a critical point of $\tilde{J}$, then
$(u,Z_{\omega}(u),\A)$ is a critical point of $J$ (and {\it
viceversa}).
\end{lemma}

From \eqref{82} we have
\begin{equation} \label{86} \int_{\R^3}{qu^2\Phi_u
dx} = \int_{\R^3}{|\nabla \Phi_u|^2 dx} + q^2
\int_{\R^3}{u^2\Phi_u^2dx},
\end{equation}
which is another way of writing \eqref{81}.

Now, by \eqref{83} and \eqref{86}, we have:
\[
\begin{aligned}
\tilde{J}(u,\A) &= J(u,Z_{\omega}(u),\A) = \frac{1}{2} \int_{\R^3}{ \big[ |\nabla u|^2 - |\nabla \phi_u|^2 + |\curl \A|^2 \big] dx} \\
&+ \frac{1}{2} \int_{\R^3}{\left[ |l\nabla \theta - q\A|^2 - (\omega-q\phi_u)^2  \right]u^2 dx} + \int_{\R^3}{W(u) dx} \\
&= \frac{1}{2} \int_{\R^3}{ \big[ |\nabla u|^2 + |\curl \A|^2 + |l\nabla \theta - q\A|^2u^2 \big] dx} + \int_{\R^3}{W(u)dx} \\
&-\frac{\omega^2}{2} \int_{\R^3}{(1-q\Phi_u)u^2 dx}.
\end{aligned}
\]
Then
\begin{equation}\label{88}
\tilde{J}(u,\A)= I(u,\A) - \frac{\omega^2}{2}K_q(u),
\end{equation}
where $I:\hat{H}^1 \times \left(\mathcal{D}^1 \right)^3\to \R$ and
$K_q:\hat H^1\to \R$ are defined as
\begin{equation}\label{I}
I(u,\A) = \frac{1}{2}\int_{\R^3}{\left(|\nabla u|^2 + |\curl \A|^2 +
|l\nabla \theta- q\A|^2u^2   \right) dx} + \int_{\R^3}{W(u) dx}
\end{equation}
and \begin{equation} \label{89} K_q(u) = \int_{\R^3}{(1-q\Phi_u)u^2
dx}. \end{equation}

Now, let us introduce the \textit{reduced energy functional},
defined as
\[
\hat{\mathcal{E}}(u,\A)= \mathcal{E}(u,Z_{\omega}(u),\A),
\]
where, as in \eqref{35},
\begin{equation}\label{90}
\begin{aligned}
\mathcal{E}(u,\phi,\A)&= \frac{1}{2}\int_{\R^3}{\left(|\nabla u|^2 +
|\nabla \phi|^2 + |\curl \A|^2
+ (|l\nabla \theta- q\A|^2 + (\omega - q\phi)^2)u^2   \right)dx} \\
&+ \int_{\R^3}{W(u)dx}.
\end{aligned}
\end{equation}
By using \eqref{86} and \eqref{83}, we easily find that
\begin{equation}
\label{91} \hat{\mathcal{E}}(u,\A) = I(u,\A) +
\frac{\omega^2}{2}K_q(u).
\end{equation}

Recalling \eqref{36} and \eqref{37}, we note that
\[
Q = q \sigma = q\omega K_q(u)
\]
represents the (electric) charge, so that, if  $u \neq 0$, we can
write
\[
\hat{\mathcal{E}}(u,\A) = I(u,\A) + \frac{\omega^2}{2}K_q(u) = I(u,\A) + \frac{\sigma^2}{2K_q(u)}.
\]
Then for any $\sigma \neq 0$, the functional $E_{\sigma,q}: (\hat{H}^1\setminus\{0\}) \times \left(\mathcal{D}^1 \right)^3\to \R$, defined by
\begin{equation} \label{92}
E_{\sigma,q}(u,\A) = I(u,\A) + \frac{\omega^2}{2}K_q(u) = I(u,\A) +
\frac{\sigma^2}{2K_q(u)}
\end{equation}
represents the energy on the configuration $(u,\omega \Phi_u,\A)$
having charge $Q = q\sigma$ or, equivalently, frequency $\omega =
\frac{\sigma}{K_q(u)}$.

The following lemma holds (see \cite[Lemma 13]{BF}):
\begin{lemma} \label{lemma13}
The functional
\[
\hat{H}^1 \ni u \mapsto K(u) = \int_{\R^3}{(1-q\Phi_u)u^2dx}
\]
is differentiable and for any $u,v \in \hat{H}^1$ we have
\begin{equation} \label{93}
K'(u)[v] = 2\int_{\R^3}(1-q\Phi_u)^2uv\,dx.
\end{equation}
\end{lemma}

Introducing $E_{\sigma,q}$ turns out to be a useful choice, as the
following easy consequence shows (see \cite[Proposition 14]{BF}):
\begin{prop} \label{prop14}
Let $\sigma \neq 0$ and let $(u,\A) \in \hat{H}^1 \times
(\mathcal{D}^1)^3$, $u \neq 0$ be a critical point of
$E_{\sigma,q}$. Then, if we set $\omega = \frac{\sigma}{K_q(u)}$,
$(u,Z_{\omega}(u),\A)$ is a critical point of $J$.
\end{prop}

Therefore, by Proposition \ref{prop14} and Theorem \ref{th10} we are
reduced to study the critical points of $E_{\sigma,q}$, which is a
functional bounded from below, since all its components are
nonnegative.

However $E_{\sigma,q}$ contains the term $\int_{\R^3}{|\curl
\A|^2}$, which is not a Sobolev norm in $\left(\mathcal{D}^1
\right)^3$. In order to avoid consequent difficulties, we introduce
a suitable manifold $V \subset \hat{H}^1 \times \left(\mathcal{D}^1
\right)^3$ in the following way: first, we set
\[
\mathcal{A}_0 := \Big\{ \mathbf{X} \in C^\infty_C(\R^3 \setminus
\Sigma, \R^3): \mathbf{X} = b(r,z)\nabla \theta; \ b \in
C^\infty_C(\R^3 \setminus \Sigma, \R) \Big\},
\]
and we denote by $\mathcal{A}$ the closure of $\mathcal{A}_0$ with
respect to the norm of $\left( \mathcal{D}^1 \right)^3$. We now
consider the space
\beq \label{forV}
V := \hat{H}^1_\sharp \times
\mathcal{A},
\eeq
and we set $U = (u,\A)\in V$ with
\[
\|U\|_V = \|(u,\A)\|_V = \|u\|_{\hat{H}^1_\sharp} + \|\A\|_{(
\mathcal{D}^1 )^3}.
\]

We need the following result, for whose proof see \cite[Lemma
15]{BF}:
\begin{lemma} \label{lemma15}
If $\A \in \mathcal{A}$, then
\[
\int_{\R^3}{|\curl \A|^2 dx} = \int_{\R^3}{|\nabla \A|^2 dx}.
\]
\end{lemma}

Working in $V$ has two advantages: first, the components $\A$ of the
elements in $V$ are divergence free, so that the term
$\int_{\R^3}{|\curl \A|^2}$ can be replaced by
$\|\A\|^2_{(\mathcal{D}^1)^3} = \int_{\R^3}{|\nabla \A|^2}$. Second,
the critical points of $J$ constrained on $V$ satisfy system
\eqref{39a}--\eqref{40a}--\eqref{41a}; namely $V$ is a ``natural
constraint'' for $J$.

\section{Proof of Theorem \ref{main}}
In this section we shall always assume that $W$ satisfies $W1$),
$W2)$, $W3)$, $W4)$ and we will show that $E_{\sigma,q}$ constrained on $V$ as in \eqref{forV}
has a minimum which is a nontrivial solution of system
\eqref{39a}--\eqref{40a}--\eqref{41a}.

We start with the following {\it a priori} estimate on minimizing
sequences, whose proof  is similar to the proof of \cite[Lemma 18]{BF}:
\begin{lemma}\label{lemma18}
For any $\sigma, q>0$, any minimizing sequence $(u_n,\A_n) \subset
V$ for $E_{\sigma,q}|_V$ is bounded in $\hat{H}^1 \times \left(
\mathcal{D}^1 \right)^3$.
\end{lemma}

\begin{prop} \label{prop21}
For any $\sigma, q > 0$ there exists a minimizing sequence $U_n =
(u_n,\A_n)$ of $E_{\sigma,q}|_V$, with $u_n \geq 0$ and which is
also a Palais--Smale sequence for $E_{\sigma,q}$, i.e.
\[
E'_{\sigma,q}(u_n,\A_n) \rightarrow 0.
\]
\end{prop}
\begin{proof}
Let $(u_n,\A_n) \subset V$ be a minimizing sequence for
$E_{\sigma,q}|_V$. It is not restrictive to assume that $u_n \geq
0$. Otherwise, we can replace $u_n$ with $|u_n|$ and we still have a
minimizing sequence (see \eqref{90}). By Ekeland's
Variational Principle (see  \cite{ekel}) we can also
assume that $(u_n,\A_n)$ is a Palais--Smale sequence for
$E_{\sigma,q}|_V$, namely we can assume that
\[
E'_{\sigma,q}|_V(u_n,\A_n) \rightarrow 0.
\]
By using the same technique used to prove Theorem 16 in \cite{11},
it follows that $(u_n,\A_n)$ is a Palais--Smale sequence also for
$E_{\sigma,q}$, that is
\[
E'_{\sigma,q}(u_n,\A_n) \rightarrow 0.
\]
\end{proof}

A fundamental tool in proving the existence result, is given by the
following
\begin{lemma}\label{soprasotto}
For any $\sigma, q>0$ and for any minimizing sequence $(u_n,\A_n)
\subset V$ for $E_{\sigma,q}|_V$, there exist positive numbers
$a_1<a_2$ such that
\[
a_1\leq \int_{\R^3}(1-q\Phi_{u_n})u_n^2dx\leq a_2\mbox{ for every
}n\in \N
\]
and
\[
a_1\leq \int_{\R^3}u_n^2dx\leq a_2\mbox{ for every }n\in \N.
\]
\end{lemma}
\begin{proof}
The upper bounds are an obvious consequence of Lemma \ref{lemma18}
and of \eqref{84}, so that we only prove the lower bounds.

Since $E_{\sigma,q}(u_n,\A_n)\to \inf_VE_{\sigma,q}$, from
\eqref{92} we immediately get that there exists $a_1>0$ such that
\[
\frac{1}{\int_{\R^3}(1-q\Phi_{u_n})u_n^2dx}\leq \frac{1}{a_1} \mbox{
for every }n\in \N,
\]
and thus all the claims follow.
\end{proof}

As a corollary of the previous Lemma, we have the following result,
whose proof is now very easy, but whose consequences are crucial:
\begin{lemma}\label{inf0}
For any $\sigma,q>0$
\[
\inf_V E_{\sigma,q}>0.
\]
\end{lemma}
\begin{proof}
Assume by contradiction that $\inf_VE_{\sigma,q}=0$. Hence, there would
exist a sequence $(u_n,\A_n)_n\subset V$ such that
$E_{\sigma,q}(u_n,\A_n)\to 0$ as $n\to \infty$. Since both $I$ and
$K_q$ are nonnegative, from \eqref{92} we get
\[
I(u_n,\A_n)\to 0 \mbox{ and } \frac{1}{K_q(u_n)}\to 0 \mbox{ as }n\to
\infty.
\]
In particular,
\[
\int_{\R^3}(1-q\Phi_{u_n})u_n^2dx\to \infty \mbox{ as }n\to \infty,
\]
and thus, by \eqref{84},
\[
\int_{\R^3}u_n^2dx\to \infty \mbox{ as }n\to \infty,
\]
a contradiction to Lemma \ref{soprasotto}.
\end{proof}

The following result, which turns out to be a crucial one, is the
only point where assumption $W4)$ is used.
\begin{lemma}\label{lemma17}
There exists $\sigma_0>0$ such that there exists
$u_0\in \hat H^1$ with
\[
E_{\sigma_0,q}(u_0,0)<m\sigma_0.
\]
Moreover, if $q\leq 1$, then $\sigma_0$ depends only on $D$ and $m$,
while, if $q>1$, then $\sigma_0$ depends on $D$, $m$ and $q$.
\end{lemma}
\begin{proof}
Let us define
\[
v(x) := \begin{cases} 1 - \sqrt{(r-2)^2 + x_3^2  } , & (r-2)^2 + x_3^2 \leq 1,\\
0, & \mbox{elsewhere} .\end{cases}
\]
We define the set $\mathit{A}_{\lambda} := \{(r,x_3) \in \R^3 \
\mbox{s.t.} \ (r-2\la)^2 + x_3^2 \leq \lambda^2 \}$ and we compute
\begin{equation} \label{ala}
|\mathit{A}_{\lambda}| = \int_{\mathit{A}_{\lambda}}{dx_1dx_2dx_3}=
4\pi^2\lambda ^3=\la ^3 |\mathit{A}_1|.
\end{equation}
Of course, $v \in \hat{H}^1_r$ and, for a future need, we also
compute
\begin{equation} \label{v}
\begin{aligned}
&\int_{\R^3}{v^2 dx} = \int_{\mathit{A}_1} { \left( 1 - \sqrt{(r-2)^2 + x_3^2} \right)^2 dx_1 dx_2 dx_3} = \frac{2}{3}\pi^2, \\
&\int_{\R^3}{v dx} = \int_{\mathit{A}_1}{ \left( 1 - \sqrt{(r-2)^2 + x_3^2} \right) dx_1 dx_2 dx_3 } = \frac{4}{3}\pi^2, \\
&\int_{\R^3}{|\nabla v|^2}dx = \int_{\mathit{A}_1}{dx_1dx_2dx_3} =
4\pi^2.
\end{aligned}
\end{equation}
Moreover, for $\ve\in(0,\ve_0)$ and $\lambda \geq 1$ we define
\[
u_{\ve,\lambda}(x) = \ve^2\lambda v\left( \frac{x}{\lambda} \right).
\]
We also choose $\ve$ and $\lambda$ such that
\begin{equation}\label{epsila}
\ve\lambda\leq1,
\end{equation}
so that $0\leq u_{\ve,\lambda}\leq\ve<\ve_0$ in $\R^3$.

Then we have
\begin{equation} \label{ga}
\begin{aligned}
E_{\sigma_{\lambda},q}(u_{\ve,\lambda},0) &= \int_{\R^3}{\left[
\frac{1}{2} |\nabla u_{\ve,\lambda}|^2 + \frac{l^2}{r^2}
\frac{u_{\ve,\lambda}^2}{2} + W(u_{\ve,\lambda}) \right]dx  } +
\frac{\sigma^2}{2K_q(u_{\ve,\lambda})} \\
& = \frac{1}{2} \int_{\R^3}{ |\nabla u_{\ve,\lambda}|^2 } +
\frac{l^2}{2}\int_{\R^3}{ \frac{u_{\ve,\lambda}^2}{r^2}} +
\frac{m^2}{2}\int_{\R^3}{u_{\ve,\lambda}^2}\\
& +\int_{\R^3}N(u_{\ve,\lambda})\,dx+
\frac{\sigma^2}{2K_q(u_{\ve,\lambda})}.
\end{aligned}
\end{equation}
Now, observe that in $\mathit{A}_{\la}$ we have
\[
r \geq 2\la - \sqrt{\la^2 - x_3^2} \geq \la,
\]
so that, thanks to \eqref{ala}, we can estimate
\begin{equation}\label{polo}
\begin{aligned}
&\int_{\R^3}{\frac{u_{\ve,\lambda}^2}{r^2}dx_1dx_2dx_3} =
\ve^4\displaystyle{\int_{\mathit{A_{\lambda}}}{ \frac{ \left(
\lambda - \la\sqrt{\left( \dfrac{r}{\lambda} - 2 \right)^2 +
 \dfrac{x_3^2}{\lambda ^2}} \right)^2 }{r^2} drdx_3}} \\
&\leq \ve^4\int_{\mathit{A}_{\lambda}}{ \frac{\left( \lambda - \sqrt{(r - 2\lambda)^2 + x_3^2} \right)^2}{\la^2} drdx_3 } \\
&\leq \ve^4\int_{\mathit{A}_{\lambda}}{ \left( \frac{ \lambda -
\sqrt{(r - 2\lambda)^2 + x_3^2}}{\la} \right)^2 drdx_3} \leq \ve^4
|\mathit{A}_{\lambda}| = 4\pi^2\ve^4\lambda ^3.
\end{aligned}
\end{equation}

By the change of variable $y = x/ \la$ we immediately get
\[
\begin{aligned}
&\int_{\mathit{A}_{\la}}{ |\nabla u_{\ve,\lambda}|^2 dx} = \ve^4\la^3 \int_{\mathit{A}_1}{ |\nabla v|^2 dx},\\
& \int_{\mathit{A}_{\la}}{(u_{\ve,\lambda})^{\vartheta} dx} =
 \ve^{2\vartheta}\la^{\vartheta + 3}\int_{\mathit{A}_1}{v^{\vartheta} dx} \quad
\forall\,\theta>0.
\end{aligned}
\]
Therefore, \eqref{v}, \eqref{ga}, \eqref{polo} and $W4)$ imply
\begin{equation}\label{ga2}
\begin{aligned}
E_{\sigma,q}(u_{\ve,\lambda},0) &\leq 2\pi^2 \ve^4\la^3
+ \frac{m^2\pi^2}{3}\ve^4\la^5 + 2\pi^2l^2\ve^4\la^3\\
&-D\ve^{2\tau}\lambda^{\tau+3} \int_{\mathit{A}_1}{ v^\tau dx} +
\frac{\sigma^2}{2K_q(u_{\ve, \la})}.
\end{aligned}
\end{equation}

Now, let us note that
\[
-\Delta \Phi_{u_{\ve,\la}}=qu_{\ve,\la}^2(1-q\Phi_{u_{\ve,\la}})\leq
qu_{\ve,\la}^2,
\]
so that, by the Comparison Principle, for every $x\in\R^3$ we have
\begin{equation}\label{e4l4}
\Phi_{u_{\ve,\la}}(x)\leq
\frac{q}{4\pi}\int_{\R^3}\frac{u_{\ve,\la}^2(x-y)}{|y|}dy=\frac{q\ve^4
\la^5}{4\pi}\int_{\R^3}\frac{v^2(y)}{|x-\la y|}dy\leq
\frac{q}{2}\ve^4 \la^4.
\end{equation}
Indeed:
\[
\begin{aligned}
\int_{\R^3}\frac{v^2(y)}{|x-\la y|}dy&\leq \int_{A_1}\frac{1}{|x-\la
y|}dy=\frac{1}{\lambda^3}\int_{A_{1/\la}}\frac{1}{|x-z|}dz\\
&=\frac{1}{\lambda^3}\int_{A_{1/\la}-x}
\frac{1}{|z|}dz\leq\frac{1}{\lambda^3}
\int_{B(0,1/\la)}\frac{1}{|z|}dz=\frac{2\pi}{\la},
\end{aligned}
\]
and \eqref{e4l4} follows.

As a consequence,
\[
\begin{aligned}
K_q(u_{\ve, \la})&=\int_{\R^3}u_{\ve, \la}^2(1-q\Phi_{u_{\ve,
\la}})\,dx\geq \int_{\R^3}u_{\ve, \la}^2(1-\frac{q^2}{2}\ve^4 \la^4)\,dx\\
&=\frac{2}{3}\pi^2(1-\frac{q^2}{2}\ve^4 \la^4)\ve^4\lambda^5.
\end{aligned}
\]
Hence, choosing
\begin{equation}\label{secondac}
\ve^4\la^4\leq 1/q^2,
\end{equation}
\eqref{ga2} becomes
\[
\begin{aligned}
E_{\sigma,q}(u_{\ve,\lambda},0) &\leq 2\pi^2 \ve^4\la^3 +
\frac{m^2\pi^2}{3}\ve^4\la^5 +
2\pi^2l^2\ve^4\la^3\\
&-D\ve^{2\tau}\lambda^{\tau+3} \int_{\mathit{A}_1}{ v^\tau dx} +
\frac{3\sigma^2}{\pi^2\ve^4\lambda^5}.
\end{aligned}
\]

Now, take
\begin{equation}\label{ugual}
\ve^4\la^5=\frac{6\sigma}{m\pi^2},
\end{equation}
so that \eqref{secondac} implies
\begin{equation}\label{speriamo}
\lambda\geq\frac{6\sigma}{m\pi^2} q^2.
\end{equation}
With this choice we find
\[
E_{\sigma,q}(u_{\ve,\lambda},0)
\leq12\frac{\sigma}{m}(1+l^2)\la^{-2}+2m\sigma-E\la^{3-3\tau/2}+\frac{m\sigma}{2},
\]
where we have set $E=D(6\sigma/m\pi^2)^{\tau/2}\int v^\tau$.

Let us show that we can find $\la\geq \max\{1,3q^2\sigma/m\pi^2\}$
(and thus $\ve\leq 1$) satisfying \eqref{epsila} and
\eqref{secondac} such that
\[
12\frac{\sigma}{m}(1+l^2)\la^{-2}+\frac{5}{2}m\sigma-E\la^{3-3\tau/2}\leq
m\sigma,
\]
that is
\begin{equation}\label{magari}
\frac{12}{m}(1+l^2)+\frac{3}{2}m\la^2-F\la^{5-3\tau/2}\leq
0,
\end{equation}
where $F=D(6/m\pi^2)^{\tau/2} \sigma^{\tau/2-1} \int v^\tau $. Also
note that $5-3\tau/2<2$, since $\tau>2$.

Indeed, we choose
\begin{equation}\label{ugual2}
\lambda\geq \frac{\sqrt{8(1+l^2)}}{m},
\end{equation}
so that we can estimate the left hand side of \eqref{magari} with
\[
\frac{12}{m}(1+l^2)+\frac{3}{2}m\la^2-F\la^{5-3\tau/2}\leq
3m\lambda^2-F\lambda^{5-3\tau/2},
\]
and the last quantity is non positive as soon as
\begin{equation}\label{disug}
\lambda\leq \left(\frac{F}{3m}\right)^{2/3(\tau-2)}.
\end{equation}
Summing up, from \eqref{epsila}, \eqref{secondac}, \eqref{ugual},
\eqref{speriamo}, \eqref{ugual2} and \eqref{disug}, we are led to
solve the following set of conditions:
\begin{eqnarray}
\frac{6\sigma}{m\pi^2}\leq \lambda \label{prima}\\
\frac{6\sigma}{m\pi^2}q^2\leq \lambda\label{seconda}\\
\frac{\sqrt{8(1+l^2)}}{m}\leq \lambda\label{terza}\\
\lambda \leq\left(\frac{F}{3m}\right)^{2/3(\tau-2)}\label{quarta}.
\end{eqnarray}

Now, if $q\leq1$, \eqref{prima} implies \eqref{seconda}. Then,
choose $\sigma$ such that
\[
\frac{6\sigma}{m\pi^2}\geq \frac{\sqrt{8(1+l^2)}}{m},
\]
i.e.
\begin{equation}\label{bassa}
\sigma\geq \frac{\pi^2\sqrt{8(1+l^2)}}{6}.
\end{equation}
Hence, from \eqref{prima} and \eqref{quarta}, we must solve
\[
\frac{6\sigma}{m\pi^2}\leq \lambda\leq
\left(\frac{G}{3m}\right)^{2/3(\tau-2)}\sigma^{1/3},
\]
where $F=G\sigma^{\frac{\tau - 2}{2}}$, so that $G$
is independent of $\sigma$.

Of course, such a choice of $\lambda$ is possible
provided that
\begin{equation}\label{piubassa}
\begin{aligned}
\sigma& \leq \left(\frac{m^{\tau -4}\pi^{2\tau -6}}{6^{\tau-3}3} D\int_{\R^3}v^\tau dx
\right)^{1/(\tau-2)}\\
& \leq \left( \frac{m^{\tau -4}\pi^{2\tau -6}}{6^{\tau-3}3} D\int_{\R^3}v^2 dx
\right)^{1/(\tau-2)}
\end{aligned}
\end{equation}

In conclusion, \eqref{bassa}, \eqref{piubassa} and \eqref{v} imply
\[
\frac{\pi^2\sqrt{8(1+l^2)}}{6}\leq
cD^{1/(\tau-2)}m^{(\tau-4)/(\tau-2)},
\]
which is true by $W4)$.

On the other hand, if $q>1$, proceeding as above, we find a suitable
$\lambda$ provided that
\[
\frac{\pi^2\sqrt{8(1+l^2)}}{6}\leq
cD^{1/(\tau-2)}m^{(\tau-4)/(\tau-2)}\frac{1}{q^3}.
\]

In any case, the lemma holds.
\end{proof}

As a consequence, we can prove the following
\begin{lemma}\label{prop20}
There exists $c>0$ and a minimizing sequence $U_n = (u_n,\A_n)
\subset V$ of $E_{\sigma_0,q}|_V$ such that
\[
\int_{\R^3}{(|u_n|^\ell +|u_n|^p)dx} \geq c > 0 \ \text{for every
}n\in \N.
\]
\end{lemma}
\begin{proof}
By Lemma \ref{lemma17} we know that there exists $\delta>0$ and
$n_0\in \N$ such that
\[
E_{\sigma_0,q}(u_n,\A_n)\leq m\sigma_0-\delta,
\]
which implies in particular that
\[
\frac{m^2}{2}\int_{\R^3}u_n^2dx+\int_{\R^3}
N(u_n)\,dx+\frac{\sigma_0^2}{2\int_{\R^3}u_n^2dx}\leq
m\sigma_0-\delta.
\]
Thus
\[
\int_{\R^3} N(u_n)\,dx\leq m\sigma_0-\delta
-\left(\frac{m^2}{2}\int_{\R^3}u_n^2dx
+\frac{\sigma^2}{2\int_{\R^3}u_n^2dx}\right)\leq -\delta,
\]
since $a/(2b)+b/(2a)\geq1$ for any $a,b>0$. Then
\[
\left|\int_{\R^3} N(u_n)\,dx\right|\geq \delta \quad \mbox{ for all
}n\geq n_0,
\]
and $W2)$ imply the claim, up to a relabelling of the sequence.
\end{proof}

By Lemma \ref{lemma18} we know that any minimizing sequence $U_n :=
(u_n,\A_n) \subset V$ of $E_{\sigma_0,q}|_V$ weakly converges (up to
a subsequence). Observe that $E_{\sigma_0,q}$ is invariant by
translations along the $z$-axis, namely for every $U \in V$ and $L
\in \R$ we have
\[
E_{\sigma_0,q}(T_LU) = E_{\sigma_0,q}(U),
\]
where
\begin{equation}\label{125}
T_L(U)(x,y,z)=U(x,y,z+L).
\end{equation}
As a consequence of this invariance, we have that $(u_n,\A_n)$ does
not contain in general a strongly convergent subsequence. To
overcome this difficulty, we will show that there exists a
minimizing sequence $(u_n,\A_n)$ of $E_{\sigma_0,q}|_V$ which, up to
translations along the $z$-direction, weakly converges to a
non--trivial limit $(u_0,\A_0)$. Eventually, we will show that
$(u_0,\A_0)$ is a critical point of $E_{\sigma,q}$ for a suitable
$\sigma>0$.

In order to proceed with this strategy, we start proving the
following weak compactness result, whose proof is an adaptation of
\cite[Proposition 22]{BF}, but whose statement is much more general:
\begin{prop} \label{prop22}
There exists a Palais--Smale sequence $U_n = (u_n,\A_n)$ of
$E_{\sigma_0,q}$ which weakly converges to $(u_0,\A_0), u_0 \geq 0$
and $u_0 \neq 0$.
\end{prop}
\begin{proof}
By Proposition \ref{prop21}, we know that there exists a minimizing
sequence $U_n= (u_n,\A_n)$ of $E_{\sigma_0,q}|_V$, with $u_n \geq 0$
and which is also a Palais--Smale sequence for $E_{\sigma_0,q}$.
Moreover, by Lemma \ref{prop20}, we know that there exists $c>0$
such that
\begin{equation} \label{130}
\|u_n\|^\ell_{L^\ell} + \|u_n\|^p_{L^p} \geq c >0 \ \mbox{for $n$ large}.
\end{equation}
By Lemma
\ref{lemma18} the sequence $\{ U_n \}$ is bounded in $\hat{H}^1
\times \left( \mathcal{D}^1 \right)^3$, so we can assume that it
weakly converges. However the weak limit could be trivial. We will
show that there is a sequence of integers $j_n$ such that $V_n :=
T_{j_n}U_n \rightharpoonup U_0 = (u_0,\A_0)$ in $H^1 \times \left(
\mathcal{D}^1 \right)^3$, with $u_0 \neq 0$, see \eqref{125}.

For any integer $j$ we set
\[
\Omega_j = \{ (x_1,x_2,x_3): j \leq x_3 < j +1 \}.
\]
In the following we denote by $c$ various positive absolute
constants which may vary also from line to line.
We have for all $n$,
\begin{equation} \label{131}
\begin{aligned}
\|u_n\|^\ell_{L^\ell} &= \displaystyle{ \sum_{j}\left( \int_{\Omega_j}{|u_n|^\ell dx} \right)^{1/\ell}
\left( \int_{\Omega_j} {|u_n|^\ell dx}  \right)^{\frac{\ell-1}{\ell}}} \\
& \leq \sup_{j}{\|u_n\|_{L^\ell(\Omega_j)}} \sum_{j}{\left( \int_{\Omega_j}{|u_n|^\ell dx }
\right)^{\frac{\ell-1}{\ell}}}   \\
&\displaystyle{ \leq c \sup_{j}{\|u_n\|_{L^\ell(\Omega_j)}}\sum_{j}{\|u_n\|^{\ell-1}_{H^1(\Omega_j)}}   } \\
&= c\sup_{j}{\|u_n\|_{L^\ell(\Omega_j)}}\|u_n\|^{\ell-1}_{H^1(\R^3)}
\leq (\mbox{since} \ \|u_n\|_{H^1(\R^3)} \ \mbox{is bounded}) \\
&\leq c \sup_{j}{\|u_n\|_{L^\ell(\Omega_j)}} \quad \mbox{ for all
}n\geq 1.
\end{aligned}
\end{equation}
In the same way we get
\begin{equation} \label{132}
\|u_n\|^p_{L^p} \leq c\sup_{j}{\|u_n\|_{L^p(\Omega_j)}} \quad
\mbox{ for all }n\geq 1.
\end{equation}
Then, by \eqref{130}, \eqref{131} and \eqref{132} it immediately
follows that, for $n$ large, we can choose an integer $j_n$ such
that
\begin{equation} \label{133}
\|u_n\|_{L^\ell(\Omega_{j_n})} + \|u_n\|_{L^p(\Omega_{j_n})} \geq c > 0.
\end{equation}
Now set
\[
\left( u'_n, \A_n' \right) = U'_n(x_1,x_2,x_3)= T_{j_n}(U_n)= U_n(x_1,x_2,x_3 + j_n).
\]
Since $(U_n')_n$ is again a minimizing sequence for
$E_{\sigma_0,q}|_V$, by Lemma \ref{lemma18} the sequence $\{ u'_n
\}$ is bounded in $\hat{H}^1(\R^3)$; then (up to a subsequence) it
weakly converges to $u_0 \in \hat{H}^1(\R^3)$. Clearly $u_0 \geq 0$,
since $u'_n \geq 0$. We want to show that $u_0 \neq 0$. Now, let
$\vp = \vp(x_3)$ be a nonnegative, $C^{\infty}$--function whose
value is $1$ for $0 < x_3 < 1$ and $0$ for $|x_3| > 2$. Then, the
sequence  $\vp u'_n$ is bounded in $H^1_0(\R^2 \times (-2,2))$, and
moreover $\vp u'_n$ has cylindrical symmetry. Then, using the
compactness result of Esteban--Lions \cite{23}, we have that, up to
a subsequence,
\begin{equation} \label{135}
\vp u'_n \rightarrow \vp u_0 \ \mbox{ in} \ L^\ell(\R^2 \times
(-2,2)), \ \mbox{ in} \ L^p(\R^2 \times (-2,2))\mbox{ and a.e. in
}\R^2 \times (-2,2).
\end{equation}

Moreover for $r=p,\ell$ we clearly have
\begin{equation} \label{136}
\|\vp u'_n\|_{L^r(\R^2 \times (-2,2))} \geq \|u'_n\|_{L^r(\Omega_0)} = \|u_n\|_{L^r(\Omega_{j_n})}.
\end{equation}
Then by \eqref{135}, \eqref{136} and \eqref{133} we have
\[
\|\vp u_0\|_{L^\ell(\R^2 \times (-2,2))} + \|\vp u_0\|_{L^p(\R^2 \times (-2,2))} \geq c > 0.
\]
Thus we have that $u_0 \neq 0$, as claimed.
\end{proof}

In order to approach the conclusion, we need
\begin{prop} [] \label{prop23}
For every $q>0$ there exists $\sigma> 0$ such that $E_{\sigma,q}$
has a critical point $(u_0,\A_0), u_0 \neq 0, u_0 \geq 0$.
\end{prop}
\begin{proof}
By Proposition \ref{prop22}, there
exists a sequence $U_n = (u_n, \A_n)$ in $V$, with $u_n \geq 0$ and
such that
\begin{equation} \label{137}
E'_{\sigma_0,q}(u_n,\A_n) \rightarrow 0
\end{equation}
and
\[
(u_n,\A_n) \rightharpoonup (u_0,\A_0) \ , u_0\geq 0, \, u_0\neq 0.
\]

We now show that there exists $\sigma> 0$ such that $U_0=
(u_0,\A_0)$ is a critical point of $E_{\sigma,q}$.

By \eqref{137}, in particular we get that
\[
dE_{\sigma_0,q}(U_n)[w,0] \rightarrow 0 \ \ \mbox{and} \ \
dE_{\sigma_0,q}(U_n)[0,\mathbf{w}] \rightarrow 0 \ \mbox{, for any}
\ (w,\mathbf{w}) \in \hat{H}^1 \times \left(C^\infty_C \right)^3.
\]
Then for any $w \in \hat{H}^1$ and $\mathbf{w} \in  \left(C^\infty_C \right)^3$ we have
\begin{equation} \label{138}
\partial_uI(U_n)[w] + \partial_u \left( \frac{\sigma_0^2}{2K_q(u_n)} \right)[w]
\rightarrow 0
\end{equation}
and
\begin{equation} \label{139}
\partial_{\A}I(U_n)[\mathbf{w}] \rightarrow 0,
\end{equation}
where $\partial_u$ and $\partial_{\A}$ denote the partial
derivatives of $I$ with respect to $u$ and $\A$, respectively. So
from \eqref{138} we get for any $w \in \hat{H}^1$,
\[
\partial_uI(U_n)[w] - \frac{\sigma_0^2K'_q(u_n)}{2 \left( K_q(u_n)
\right)^2}[w] \rightarrow 0,
\]
which can be written as follows:
\begin{equation} \label{140}
\partial_uI(U_n)[w] - \frac{\omega_n^2 K'_q(u_n)}{2}[w] \rightarrow 0,
\end{equation}
where we have set
\[
\omega_n = \frac{\sigma_0}{K_q(u_n)}.
\]
By Lemma \ref{soprasotto} we have that (up to a subsequence)
\[
\omega_n \rightarrow \omega_0 > 0.
\]
Then by \eqref{140} we get for any $w \in \hat{H}^1$
\begin{equation} \label{141}
\partial_uI(U_n)[w] - \frac{\omega_0^2 K'_q(u_n)}{2}[w] \rightarrow 0.
\end{equation}
Now, let $\Phi_n$ be the solution in $\mathcal{D}^1$ of the equation
\begin{equation} \label{142}
- \Delta \Phi_n + q^2 u_n^2\Phi_n = q u_n^2.
\end{equation}
Since $\{ u_n \}$ is bounded in $H^1$ and since $\Phi_n$ solves
\eqref{142}, by \eqref{stima} we have that $\{ \Phi_n \}$ is bounded
in $\mathcal{D}^1$ and, checking with test functions in
$C^\infty_C(\R^3)$, it is easy to see that (up to a subsequence) its
weak limit $\Phi_0$ is a weak solution of
\begin{equation} \label{143}
- \Delta \Phi_0 + q^2 u_0^2\Phi_0 = q u_0^2.
\end{equation}
Moreover, by Lemma \ref{lemma13}, we have
\begin{equation} \label{144}
K'_q(u_n)[w] = 2\int_{\R^3}u_nw(1-q\Phi_n)^2 dx\ \ \mbox{and} \ \ K'_q(u_0)[w]= 2\int_{\R^3} u_0w(1-q\Phi_0)^2dx
\end{equation}
for every $w\in \hat H^1$.

We claim that
\begin{equation} \label{146}
K'_q(u_n)[w] \rightarrow K'_q(u_0)[w] \mbox{  for any $w \in
\hat{H}^1$}.
\end{equation}
Indeed, by \eqref{denso}, for any $w\in \hat H^1$ and every $\ve>0$,
there exists $w_\ve\in C^\infty_C\cap \hat H^1$ such that
$\|w-w_\ve\|_{\hat H^1}<\ve$. Then,
\[
\begin{aligned}
K'_q(u_n)[w]-K'_q(u_0)[w]&=K'_q(u_n)[w-w_\ve]\\
&+[K'_q(u_n)-K'_q(u_0)][w_\ve]-K'_q(u_0)[w_\ve-w].
\end{aligned}
\]
But the sequence of operators $(K'(u_n))_n$ is bounded in $(\hat
H^1)'$, while $[K'_q(u_n)-K'_q(u_0)][w_\ve]\to 0$ by the Rellich
Theorem. The claim follows.

Similar estimates show that for any $w \in \hat{H}^1$
\begin{equation} \label{147}
\partial_uI(U_n)[w] \rightarrow \partial_uI(U_0)[w].
\end{equation}
Then, passing to the limit in \eqref{141}, by \eqref{146} and
\eqref{147}, we get
\begin{equation} \label{148}
\partial_uI(U_0)[w] - \frac{\omega_0^2 K'_q(u_0)}{2}[w] = 0 \ \mbox{for
any} \ w \in \hat{H}^1.
\end{equation}
On the other hand, similar arguments show that we can pass to the
limit also in $\partial_{\A}I(U_n)[\mathbf{w}]$ and have
\begin{equation} \label{149}
\partial_{\A}I(U_n)[\mathbf{w}] \rightarrow \partial_{\A}I(U_0)[\mathbf{w}] \mbox{ for all } \ \mathbf{w} \in  \left(C^\infty_C \right)^3.
\end{equation}
From \eqref{139} and \eqref{149} we get
\[
\partial_{\A}I(U_0)[\mathbf{w}] = 0 \ \mbox{for all} \ \mathbf{w} \in
 \left(C^\infty_C \right)^3,
\]
and, by density, for any $ \mathbf{w} \in
 \left(\mathcal{D}^1 \right)^3$. From \eqref{148} we thus deduce that $U_0 = (u_0,\A_0)$
is a critical point of $E_{\sigma,q}$ with $\sigma= \omega_0K_q(u_0)
> 0$.
\end{proof}

Now we are ready to prove the main existence Theorem \ref{main}.
\begin{proof}[Proof of Theorem $\ref{main}$]
The first part of Theorem \ref{main} immediately follows from
Propositions  \ref{prop14}, \ref{prop23} and Theorem \ref{th10}. In
fact, if the couple $(u_0, \A_0)$ is like in Proposition
\ref{prop23}, by Proposition \ref{prop14} and Theorem \ref{th10} we
deduce that $(u_0,\omega_0,\phi_0,\A_0)$ with $\omega_0 =
\frac{\sigma}{K_q(u_0)}, \phi_0 = Z_{\omega_0}(u_0)$, solves
\eqref{39a}--\eqref{40a}--\eqref{41a}.

Now assume $q=0$, then, by \eqref{40a} and \eqref{41a}, we easily
deduce that $\phi_0 = 0$ and $\A_0 = 0$. Finally assume that $q >
0$. Then, since $\omega_0 > 0$, by \eqref{40a} we deduce that
$\phi_0 \neq 0$. Moreover by \eqref{41a} we deduce that $\A_0 \neq
0$ if and only if $l \neq 0$.
\end{proof}

\section{Solutions with full probability}\label{secfissa}

Throughout this section we are concerned with a different approach
to system \eqref{39a}--\eqref{40a}--\eqref{41a}: namely, we look for
solutions having full probability and we prove Proposition
\ref{vincolo}. From a physical point of view such solutions are the
most relevant ones, and in general they cannot be obtained from the
solutions found in Theorem \ref{main} by a rescaling argument,
unless some homogeneity in the potential is given. However, this is not
the case if $N\neq 0$.

Therefore, we will work in the new manifold $\tilde V := V \cap
\mathcal{S}$, where
\[
\mathcal{S}=\Big\{(u,\A) \in V\,:\, \int_{\R^3}u^2 dx =1\Big\}.
\]

We follow the lines of the previous part of the paper, and for this
reason we will be sketchy, though some differences will appear. For
example, we begin with the following
\begin{prop} \label{prop21new}
For any $\sigma, q \geq 0$ there exists a minimizing sequence $U_n =
(u_n,\A_n)$ of $E_{\sigma,q}|_{\tilde V}$, with $u_n \geq 0$, and a
sequence $(\mu_n)_n \in \R$,
such that
\[
E'_{\sigma,q}(u_n,\A_n)(v,\bs{B}) - \mu_n\int_{\R^3}u_nv\,dx
\rightarrow 0 \quad \forall \, (v,\bs{B})\in \tilde{V}.
\]
Moreover, $(\mu_n)_n$ converges to some  $\mu \in \R$ as $n \rightarrow \infty$.
\end{prop}

\begin{proof}
Let $(u_n,\A_n) \subset V$ be a minimizing sequence for
$E_{\sigma,q}|_{\tilde V}$. Working with $u_n \geq
0$, or replacing $u_n$ with $|u_n|$ if necessary, we still have a
minimizing sequence (see \eqref{90}). By Ekeland's
Variational Principle we can also
assume that $(u_n,\A_n)$ is a Palais--Smale sequence for
$E_{\sigma,q}|_{\tilde V}$, namely we can assume that
\[
E'_{\sigma,q}|_{\tilde V}(u_n,\A_n) \rightarrow 0,
\]
i.e. there exists a sequence $(\mu_n)_n \in \R$ with
\begin{equation}\label{miserve}
E'_{\sigma,q}(u_n,\A_n)(v,\bs{B}) - \mu_n \int_{\R^3}u_n v dx
\rightarrow 0, \ \forall \ v \in {\hat H}^1_\sharp, \ \forall \,
\bs{B} \in \mathcal{A}.
\end{equation}
Taking $(u_n,\A_n)$ as a test function and using $\int_{\R^3}u_n^2dx
= 1$ for all $n\in \N$, we get
\begin{equation} \label{new}
E'_{\sigma,q}(u_n,\A_n)(u_n,\A_n) - \mu_n \int_{\R^3}u_n^2dx =
E'_{\sigma,q}(u_n,\A_n)(u_n,\A_n) - \mu_n  \rightarrow 0.
\end{equation}
From \eqref{new} we get
\begin{equation} \label{new2}
\begin{aligned}
\mu_n &=E'_{\sigma,q}(u_n,\A_n)(u_n,\A_n)+o(1)  \\
 &=\int_{\R^3}
|\nabla u_n|^2 dx + \int_{\R^3} |l\nabla \theta - q \A_n|^2u_n^2 dx
+ \int_{\R^3} W'(u_n)u_n dx \\  &+\int_{\R^3} |\curl \A_n|^2  dx +
q\int_{\R^3}u_n^2|\A_n|^2 dx + \sigma K'_q (u_n)u_n dx+o(1),
\end{aligned}
\end{equation}
where $o(1)\to 0$ as $n\to \infty$. Thus, since all the terms in the
right-hand-side of \eqref{new2} are bounded, as already shown for
Lemma \ref{lemma18}, we get that also $(\mu _n)_n$ is bounded;
hence, there exists $\mu\in \R$ such that, up to a subsequence,
$\mu_n\to\mu$ as $n\to \infty$.
\end{proof}

Now we restate Proposition \ref{prop22} which still holds in this
case thanks to Proposition \ref{prop21new}, hence we get
\begin{prop} \label{prop22new}
There exists a Palais--Smale sequence $U_n = (u_n,\A_n)$ of
$E_{\sigma_0,q}$ which weakly converges to $(u_0,\A_0), u_0 \geq 0$
and $u_0 \neq 0$.
\end{prop}

In order to prove Proposition \ref{vincolo} we should just notice
that the analogue of Proposition \ref{prop23} still holds using
Proposition \ref{prop21new} and Proposition \ref{prop22new}. Hence,
we just restate the result of Proposition \ref{prop23} as follows:
\begin{prop} \label{prop23new}
For every $q>0$ there exists $\sigma> 0$ such that $E_{\sigma,q}$
has a critical point $(u_0,\A_0), u_0 \neq 0, u_0 \geq 0$.
\end{prop}

Finally, we conclude with the
\begin{proof}[Proof of Proposition $\ref{vincolo}$]
It is a natural consequence of what already proved, exactly as done
for the proof of Theorem \ref{main} in the previous section. Namely,
since Proposition \ref{prop23new} holds by Proposition
\ref{prop21new} and Proposition \ref{prop22new}, we can conclude
that our claim is true thanks to Propositions \ref{prop23new},
\ref{prop14} and Theorem \ref{th10}.

Now, suppose that $\omega^2 \leq m^2$ and $ N'(s)s \geq 0$. Passing
to the limit as $n \rightarrow \infty$ in \eqref{miserve} with
$v=u_0$ and $\bs{B}=\bs0$, as in the proof of Proposition
\ref{prop23}, we get
\[
\int_{\R^3}[ |\nabla u_0|^2 + |l\nabla \theta - q\A_0|^2u_0^2 -
(\omega - q\phi_{u_0})^2u_0^2 + W'(u_0)u_0] dx = \mu
\int_{\R^3}u_0^2 dx,
\]
which can be written as
\[
\begin{aligned}
&\int_{\R^3}[ |\nabla u_0|^2 + |l\nabla \theta - q\A|^2u^2_0 + (m^2
- \omega^2)u^2_0 - (q\phi - 2\omega)u^2_0q\phi_{u_0}\\
& + N'(u_0)u_0] dx = \mu\int_{\R^3}u_0^2 dx.
\end{aligned}
\]
Thanks to \eqref{83}, \eqref{84} and to the hypotheses under
consideration, we get $\mu >0$, so that the effective mass (see
Definition \ref{massa}) is strictly less than the original mass.
\end{proof}

\section{Non-existence of standing solutions}
In this section we shall prove Theorem \ref{senzanome}. To this
purpose, we re--write the usual system using \eqref{n1a}, so that we
deal with
\begin{align}
&-\Delta u + \left[ |l\nabla \theta - q\A|^2 + m^2 - (\omega - q \phi)^2 \right ]u + N'(u) = 0, \label{ne1} \\
&-\Delta \phi = q(\omega - q\phi)u^2, \label{ne2} \\
&\curl (\curl \A) = q(l\nabla \theta - q\A)u^2. \label{ne3}
\end{align}

\begin{proof}[Proof of Theorem $\ref{senzanome}$]
If $\bs{A}=\bs{0}$, in \cite{tdnonex} a variational identity for
solutions of \eqref{ne1} was given. However, the same identity holds
when $\A\neq \bs{0}$, and it reads as follows:
\begin{equation}\label{v1}
\begin{aligned}
0=&-\int_{\R^3} |\nabla u|^2 dx + \int_{\R^3} |\nabla \phi|^2 dx -
3\Omega \int_{\R^3}u^2 dx \\
&- 3q\int_{\R^3}(2\omega -q\phi)\phi u^2 dx + 6\int_{\R^3}F(u) dx,
\end{aligned}
\end{equation}
where we have set $\Omega = m^2 - \omega ^2$,
$F(s)=\int_0^sf(t)\,dt$ and
\[
f(u)=-|l \nabla \theta-q\A|^2u-N'(u).
\]

Since $\phi$ solves \eqref{ne2}, we have
\begin{equation}\label{v2}
\int_{\R^3} |\nabla \phi|^2 dx = q \int_{\R^3}(\omega -q\phi)u^2
\phi dx;
\end{equation}
substituting \eqref{v2} into \eqref{v1} and computing $F(u)$ we get
\begin{equation} \label{v3}
\begin{aligned}
0 =&-\int_{\R^3} |\nabla u|^2 dx-\int_{\R^3}\left[ 3\Omega+5q\omega
\phi-2q^2\phi^2+3 |l\nabla\theta - q\A|^2
\right]u^2dx\\
&-6\int_{\R^3}N(u)\,dx.
\end{aligned}
\end{equation} By
\eqref{83} and \eqref{84}, if $N\geq 0$ and  $\omega^2 < m^2$, we
get $u\equiv 0$.

Moreover, since $u$ solves \eqref{ne1}, we have
\begin{equation} \label{ne4}
\int_{\R^3}|\nabla u|^2 dx + \int_{\R^3} |l\nabla\theta - q\A|^2u^2
dx + m^2\int_{\R^3}u^2 dx - \int_{\R^3}(\omega - q\phi)^2u^2 +
\int_{\R^3}N'(u)u dx =0;
\end{equation}
substituting the expression $\int_{\R^3}|\nabla u|^2dx$ taken from
\eqref{ne4} into \eqref{v3}, we obtain
\begin{equation} \label{ne5}
\begin{aligned}
0 =& q\int_{\R^3}(q\phi-3\omega) u^2 \phi\, dx -
2\int_{\R^3}|l\nabla\theta - q\A|^2u^2 dx \\
&+ 2(\omega^2 -m^2)\int_{\R^3}u^2 dx+ \int_{\R^3}[N'(u)u - 6N(u)]dx.
\end{aligned}
\end{equation}
Thanks to \eqref{83} and \eqref{84}, all the terms in \eqref{ne5}
are non--positive if $\omega^2<m^2$, $N'(s)s - 6N(s) \leq 0$; hence
$u\equiv 0$.

Finally, when $N'(s)s\geq 2N(s)$, we proceed as follows: from
\eqref{ne4} we get
\begin{equation} \label{ne6}
\begin{aligned}
\Omega\int_{\R^3}u^2 dx = &- \int_{\R^3}|\nabla u|^2 dx -
\int_{\R^3} |l\theta - q\A|^2u^2 dx\\
&- 2q\omega \int_{\R^3}u^2 \phi
dx + q^2\int_{\R^3} u^2 \phi^2 dx - \int_{\R^3} N'(u)u\, dx.
\end{aligned}
\end{equation}
Substituting \eqref{ne6} into \eqref{v3} we get
\begin{equation} \label{ne7}
0 = 2\int_{\R^3}|\nabla u|^2 dx + \int_{\R^3} qu^2\phi (\omega - q
\phi) dx + \int_{\R^3} [3N'(u)u - 6N(u)]\, dx.
\end{equation}
Analogously, now all the coefficients are non--negative, and thus
$u\equiv0$.
\end{proof}

\end{document}